\documentclass[11pt]{amsart}

\usepackage{amssymb, amsmath, textcomp, amsthm, amsfonts, bm}
\usepackage{bbm,dsfont}
\usepackage{color}
\usepackage{cancel}
\usepackage{hyperref}

\usepackage{lipsum}
\usepackage{comment}

\title[A restriction estimate]{Fourier restriction estimates for surfaces of co-dimension two in $\R^5$}
\author{Shaoming Guo and Changkeun Oh}
\date{}

\def\R{\mathbb{R}}

\def\C{\mathbb{C}}
\def\nint{\mathop{\diagup\kern-13.0pt\int}}
\def\Z{\mathbb{Z}}

\def\S{\mathcal{S}}

\def\beq{\begin{equation}}
\def\endeq{\end{equation}}
\def\bg{\begin{gathered}}
\def\eg{\end{gathered}}

\def\mc{\mathcal}
\def\lesim{\lesssim}

\def\rank{\text{rank}}

\newcommand{\norm}[1]{ \left|  #1 \right| }
\newcommand{\Norm}[1]{ \left\|  #1 \right\| }

\newcommand{\Rone}{\text{R}_1}
\newcommand{\Rzero}{\text{R}_0}
\newcommand{\Rtwo}{\text{R}_2}
\newcommand{\Rd}{\text{R}_{d'}}
\def\P{\mathcal{P}}
\def\L{\mathbb{L}}
\def\dec{\text{Dec}}

\numberwithin{equation}{section}
\theoremstyle{plain}
\newtheorem{thm}{Theorem}[section]
\newtheorem{prop}[thm]{Proposition}

\newtheorem{lem}[thm]{Lemma}

\newtheorem{defi}[thm]{Definition}

\newtheorem{rem}[thm]{Remark}

\newtheorem*{conj*}{Conjecture}

\newtheorem*{openproblem*}{Open Problem}

\begin{document}
\maketitle

\begin{abstract}
   We prove $L^p \rightarrow L^q$ Fourier restriction estimates for 3-dimensional quadratic surfaces in $\R^5$.  Our results are sharp, up to endpoints, for a few classes of surfaces. 
\end{abstract}

\section{Introduction}

Let $d\ge 2$ and $P(\xi), Q(\xi)$ be two real quadratic forms depending on $\xi\in \R^d$. Denote by $P$ and $Q$ the Hessian matrices of $P(\xi)$ and $Q(\xi)$. We say that the pair $(P, Q)$ is irreducible if $P$ and $Q$ are linearly independent and there does not exist a linear change of coordinates in $\xi$ after which both $P(\xi)$ and $Q(\xi)$ omit a common variable. Otherwise, we say that $(P, Q)$ is reducible. 

Consider the $d$-dimensional surface
\begin{equation}\label{0822_surface}
    S=\{(\xi,P(\xi),Q(\xi)): \xi \in [-1,1]^d \}
\end{equation}in $\R^{d+2}$. Let $E^Sf$ denote the extension operator associated to $S$ defined by
\begin{equation}\label{0823equ_extension}
    E^Sf(x',x_{d+1},x_{d+2}):=\int_{[-1,1]^d}f(\xi)e(\xi \cdot x'+P(\xi)x_{d+1}+Q(\xi)x_{d+2})\,d\xi.
\end{equation}
Here, $x'=(x_1,\ldots,x_d)$ and $e(t)=e^{2\pi it}$.

We are interested in the following Fourier restriction problem: Find all the pairs $(p,q)$ such that 
\begin{equation}\label{restrictionestimate}
    \|E^Sf\|_{L^q(\R^{d+2})} \leq C_{p,q,S, d}\|f\|_{L^p(\R^d)}
\end{equation}
for every function $f$. The main result of the paper is as follows. 

\begin{thm}\label{0822main_thm}
Let $d=3$. For every pair $(P, Q)$ that is irreducible, the restriction estimate \eqref{restrictionestimate} holds with $q>4$ and $p>q/(q-3)$. 
\end{thm}

It is a coincidence that our methods in the current paper produced the same range of $(p, q)$ for every irreducible pair $(P, Q)$. Indeed, within the class of all irreducible pairs, there are a few smaller classes of pairs which admit very different Fourier restriction phenomenon from each other. Moreover, we will also see that the proofs of \eqref{restrictionestimate} for different classes of pairs of $(P, Q)$ also differ significantly.

We proceed with a classification of irreducible pairs $(P, Q)$. We say that two pairs of quadratic forms $(P, Q)$ and $(P', Q')$ are equivalent and denote by $(P,Q) \equiv (P',Q')$ if there exist two invertible real matrices $M_1\in M_{3\times 3}$ and $M_2\in M_{2\times 2}$ such that 
\begin{equation}
    (P'(\xi), Q'(\xi))^T=M_2\cdot (P(M_1\cdot \xi), Q(M_1\cdot \xi))^T, \text{ for every } \xi\in \R^3.
\end{equation}
We say that the pair $(P, Q)$ satisfies the (CM) condition if 
\begin{equation}\label{curvature_zz}
    \int_{S^1}|\det{(x_1P+x_2Q )}|^{-\gamma}\,d\sigma(x_1,x_2) < \infty\;\; \text{for every} \;\; 0<\gamma<2/3,
\end{equation}
where $S^1$ denotes the unit circle on $\R^2$ and $d\sigma$ the arc-length measure on $S^1$. This condition was introduced by Christ \cite{christthesis, MR766216} and Mockenhaupt \cite{Mockenhaupt} where they proved that the pair $(P, Q)$ satisfies the (CM) condition if and only if \eqref{restrictionestimate} holds for $p=2, q=14/3$. We say that the pair $(P, Q)$ satisfies the (D) condition if 
\begin{equation}\label{equ_determinant_condition}
    \det [\nabla P(\xi), \nabla Q(\xi), w]\not\equiv 0,
\end{equation}
for any non-zero vector $w\in \R^3$. In other words, the left hand side of \eqref{equ_determinant_condition}, when viewed as a polynomial in $\xi$, does not vanish identically. For $0\le d'\le 2$, we say that the pair $(P, Q)$ satisfies the $(\Rd)$ condition if 
\begin{equation}\label{equ_rank_condition}
    \min_{\text{ hyperplane } H\subset \R^3} \max_{\lambda_1, \lambda_2\in \R} \rank\Big( (\lambda_1 P(\cdot)+\lambda_2 Q(\cdot))|_{H}\Big)=d'. 
\end{equation}
The (D) condition and the $(\Rd)$ condition were introduced in \cite{MR3945730} and \cite{arxiv:1902.03450} to study decoupling inequalities for surfaces of co-dimension two. 

\begin{thm}\label{thm_classification_of_surfaces}
\hfill
\begin{itemize}
\item[(1)] The (CM) condition is equivalent to the (D) condition together with the $(\Rtwo)$ condition. 
\item[(2)] If $(P, Q)$ fails the (D) condition, then it is equivalent to $(\xi_1^2\pm\xi_2^2, \xi_3^2)$ or $(\xi_1\xi_2+\xi_3^2, \xi_1^2)$. 
\item[(3)] If $(P, Q)$ satisfies the (D) condition and the $(\Rone)$ condition, then it is equivalent to $(\xi_1 \xi_2, \xi_1\xi_3+\xi_2^2)$ or $(\xi_1\xi_2, \xi_1^2\pm\xi_3^2)$. 
\item[(4)] If $(P, Q)$ satisfies the (D) condition and the $(\Rzero)$ condition, then it is equivalent to $(\xi_1 \xi_2, \xi_1\xi_3)$. 
\end{itemize}
\end{thm}

\begin{thm}\label{thm_sharpness_degenerate}
Theorem \ref{0822main_thm} is sharp, up to endpoints, for all pairs $(P, Q)$ that are equivalent to those listed in item (2) or item (4) of Theorem \ref{thm_classification_of_surfaces}.  In other words, for all these pairs,  \eqref{restrictionestimate} fails for $q<4$ or $q\ge 4, p<q/(q-3)$. 
\end{thm}

\begin{rem}
Among all $3$-surfaces in $\R^5$, the ones that satisfy the (CM) condition are in some sense the most non-degenerate ones. Regarding these surfaces, Christ \cite{christthesis} showed that \eqref{restrictionestimate} fails for $q\le 10/3$, even if we take $p=\infty$ there. 
\end{rem}

Fourier restriction estimates for quadratic surfaces of co-dimension bigger than one have been studied by a number of authors in the literature. Previously we have mentioned Christ \cite{christthesis, MR766216} and Mockenhaupt \cite{Mockenhaupt}. Other than these, Oberlin \cite{MR2078263} studied Fourier restriction estimates and $L^p \rightarrow L^q$ improving estimates for the surface $(\xi_1, \xi_2, \xi_3, \xi_1^2+\xi_2^2, \xi_2^2+\xi_3^2)$, which is a surface that satisfies the (CM) condition; Banner \cite{nla.cat-vn3284359} studied restriction estimates for certain surfaces not satisfying the (CM) condition; Bak and Lee \cite{MR2064058}, and Oberlin \cite{articleoberlin} obtained sharp Fourier restriction estimates for certain quadratic moment surfaces of high co-dimensions; Bak, Lee and Lee \cite{MR3694011} studied bilinear Fourier restriction estimates for surfaces of co-dimension bigger than one and deduce a linear restriction estimate for the complex paraboloid in $\C^3$, which can be identified with $(\xi_1,\xi_2,\xi_3,\xi_4,\xi_1^2-\xi_2^2+\xi_3^2-\xi_4^2,\xi_1\xi_2+\xi_3\xi_4)$; Lee and Lee \cite{2019arXiv190304093L} studied restriction estimates for complex hyper-surfaces, which can also be viewed as special cases of real surfaces of co-dimension two, in some high dimensions. 

\medskip

\noindent {\bf Outline of paper.} In Section \ref{200914section2} we present some linear algebra computation and provide the proof of Theorem \ref{thm_classification_of_surfaces}, a classification of surfaces of co-dimension two in $\R^5$; in Section \ref{200914section3} we  prove Theorem \ref{thm_sharpness_degenerate} by constructing a few explicit examples; in Section     \ref{200906section:reduction}, we present an epsilon removal lemma, which reduces \eqref{restrictionestimate} to some ``local" estimate with an epsilon loss; the proof of Theorem \ref{0822main_thm} will be presented in the remaining sections. 

The main novelty of the paper lies in the proof of Theorem \ref{0822main_thm} for surfaces that fail the (D) condition, that is, surfaces that are equivalent to $(\xi_1, \xi_2, \xi_3, \xi_1^2, \xi_1\xi_2+\xi_3^2)$.  The proof is presented in the last section. It relies on a nested induction argument. More precisely, we will repeatedly apply a broad-narrow analysis of Bourgain and Guth \cite{MR2860188}  combined with certain bilinear restriction estimate with favourable dependence on the supports of input functions (See Lemma \ref{200901lem7.4}). As a consequence of  the iteration, we reduce our restriction estimate to a restriction estimate for lower dimensional surfaces (See Proposition \ref{lowerdim}), which is also proved in the last section.

\medskip

\noindent {\bf Notation. } For a set $A\subset \R^d$ and a function $f$ defined on $\R^d$, we use $1_A$ to denote the indicator function of $A$ and $f_A$ to denote $f\cdot 1_A$. Define $E_A^S f$ to be $E^S f_A$. Let $K_1\le K$ be two dyadic numbers, and $\Box\subset [0, 1]^d$ be a dyadic cube of side length $1/K_1$. We use $\mc{P}(\Box, 1/K)$ to denote the partition of $\Box$ into dyadic cubes of side length $1/K$. If $\Box=[0, 1]^d$, then $\mc{P}(\Box, 1/K)$ is abbreviated to $\mc{P}(1/K)$.

\subsection*{Acknowledgements} The authors were supported in part by the NSF grant 1800274. The authors would like to thank the referee for valuable suggestions.

\section{Some linear algebra: Proof of Theorem \ref{thm_classification_of_surfaces}}\label{200914section2}

Let us first prove items (2)--(4). We will show that if a pair $(P, Q)$ fails either the (D) condition or the $(\Rtwo)$ condition, then it must be equivalent to one of the pairs in items (2)--(4). 

\medskip

First of all, let us mention that item (2) and item (4) have already been obtained in \cite{2019arXiv191203995G}, see item (3) and item (2) of Theorem 1.1 there, respectively. Let us be more precise. Recall the condition (1.8) in \cite{2019arXiv191203995G}: 
\begin{equation}
    P \text{ and } Q \text{ do not share a common linear real factor.} \tag{$\text{R}_{\ge 1}$}
\end{equation}
Here we call it the condition $(\text{R}_{\ge 1})$, as it is equivalent to the union of the condition $(\Rone)$ and the condition $(\Rtwo)$. Moreover, as we always assume that the pair $(P, Q)$ is irreducible, the (D) condition and the $(\Rzero)$ condition can not fail simultaneously. This explains why item (3) and item (2) of Theorem 1.1 in \cite{2019arXiv191203995G} are equivalent to item (2) and item (4) of the above Theorem \ref{thm_classification_of_surfaces}, respectively. 

\medskip

From now on, we consider pairs $(P, Q)$ that satisfy the (D) condition and the $(\Rone)$ condition and show that they are equivalent to examples listed in item (3). We pick a hyperplane $H$ such that the equality in \eqref{equ_rank_condition} is achieved with $d'=1$. By linear transformations, we may assume that $H$ is the plane $\xi_1=0$ and 
\begin{equation}
    (P,Q) \equiv (\xi_1 L_1(\xi_1, \xi_2,\xi_3)+\xi_2^2, \xi_1 L_2(\xi_1, \xi_2,\xi_3))
\end{equation}
where $L_1$ and $L_2$ are linear forms. 

There are two subcases depending on whether $L_2$ depends on $\xi_3$ or not. Suppose that it does. By a change of variables in $\xi_1, \xi_2$ and $\xi_3$, we obtain 
\begin{equation}
\begin{split}
(P,Q) &\equiv (a\xi_1^2+b\xi_1\xi_2+\xi_2^2,\xi_1\xi_3) \equiv
(a'\xi_1^2+\xi_2^2,\xi_1\xi_3),
\end{split}
\end{equation}
for some real numbers $a, b$ and $a'$.
By scaling, we may assume that $a'=0,1,-1$. We can rule out the case that $a'=0$ as otherwise the pair $(P, Q)$ violates the (D) condition. The other two possibilities $a=\pm 1$ are listed in item (3). This finishes the discussion on the case that $L_2$ depends on $\xi_3$. Let us now suppose that $L_2$ does not depend on $\xi_3$. Then we have 
\begin{equation}
    (P,Q) \equiv (\xi_1 L_1(\xi_1, \xi_2,\xi_3)+\xi_2^2, \xi_1 \xi_2),
\end{equation}
as we know that $L_2$ can not vanish and the (D) condition guarantees that $L_2$ depends on $\xi_2$. By adding a multiple of $Q$ to $P$, we further obtain 
\begin{equation}
    (P,Q) \equiv (a\xi_1^2+b\xi_1 \xi_3+\xi_2^2, \xi_1 \xi_2),
\end{equation}
for some $a, b\in \R$. In particular, we know $b\neq 0$ as otherwise the pair becomes reducible. By a further change of variables in $\xi_1$ and $\xi_3$, we obtain that the pair $(P, Q)$ is equivalent to $(\xi_1\xi_3+\xi_2^2, \xi_1\xi_2)$, which was also listed in item (3). This finishes the proof of item (3). 

\medskip

In the end, let us prove item (1). We start with proving that if a pair $(P, Q)$ satisfies the (CM) condition, then it also satisfies the (D) condition and the $(\Rtwo)$ condition. We argue by contradiction and suppose that $(P, Q)$ fails one of these two conditions. Previously we have proven that $(P, Q)$ must be equivalent to one of the pairs in items (2)--(4). Direct computations show that the (CM) condition is violated for all those pairs, and therefore we arrive at a contradiction. 

Next we show the converse direction of the equivalence, that is, we assume that $(P, Q)$ satisfies the (D) condition and the $(\Rtwo)$ condition, and aim to show that it also satisfies the (CM) condition. We will again argue by contradiction and assume that the (CM) condition fails. Define $D(x_1,x_2):=\det(x_1P+x_2Q)$. 

First of all, under the assumption that $(P, Q)$ satisfies the (D) condition and the $(\Rtwo)$ condition, we claim that if $D(x_1,x_2)$ is identically zero, then $(P,Q) \equiv (\xi_1\xi_2,\xi_1\xi_3)$. Note that the surface $(\xi_1\xi_2,\xi_1\xi_3)$ violates the $(\Rtwo)$ condition, and therefore we arrive at a contradiction. Hence, by the claim, we may assume that $D(x_1,x_2)$ is a nonzero homogeneous polynomial of degree three. Let us give a proof of the claim. Without loss of generality, we may assume that $\mathrm{rank}(P) \geq \mathrm{rank}(Q)$. It is easy to see that the $(\Rtwo)$ condition  is violated when the rank of $P$ is one. Also, for the case that the rank of $P$ is three, the function $D(x_1,x_2)$ can not be identically zero. Let us consider the remaining case that the rank of $P$ is two. By a change of variables, we may assume that $P=\xi_1^2+\xi_2^2$ or $\xi_1\xi_2$. Let us take a look at the former case first. Since $(P,Q)$ is irreducible, $Q$ depends on $\xi_3$. We may assume that $Q$ does not contain $\xi_3^2$. Otherwise, $D(x_1,x_2)$ can not be identically zero. Write
\begin{equation}
    (P,Q) \equiv (\xi_1^2+\xi_2^2,a\xi_1^2+b\xi_1\xi_2+c\xi_1\xi_3+d\xi_2\xi_3).
\end{equation}
By routine calculations, we see that $D(x_1,x_2)$ is identically zero only if $c=d=0$. However, this cannot happen because $(P,Q)$ is irreducible. This finishes the discussion on the former case. Let us move on to the latter case that $P=\xi_1\xi_2$. Similarly, we may assume that $Q$ does not contain $\xi_3^2$. Write
\begin{equation}
    (P,Q) \equiv (\xi_1\xi_2,a\xi_1^2+b\xi_2^2+c\xi_1\xi_3+d\xi_2\xi_3).
\end{equation}
Without loss of generality, we may assume that $d=1$ and by a change of variables, we may further assume that $b=0$. The function $D(x_1,x_2)$ now becomes $2y^2(cx-ay)$, which is identically zero if and only if $c=a=0$. Hence, $(P,Q) \equiv (\xi_1\xi_2,\xi_1\xi_3)$ and this completes the proof of the claim. 

We have shown that $D(x_1,x_2)$ is a nonzero homogeneous polynomial of degree three.
From the definition \eqref{curvature_zz} of the (CM) condition and the fact that it is violated, we see immediately that $D(x_1,x_2)$ must admit a linear factor of multiplicity at least two. By rotation, we may assume that $D(x_1,x_2)=x_2^2 L(x_1,x_2)$ for some linear form $L$. Note that
\begin{equation}
    \det(P)=D(1,0)=0.
\end{equation}
After diagonalizing the matrix $P$, we may assume that $P(\xi)=\xi_1^2+a\xi_2^2$ with $a=0, \pm 1$. We can exclude the case that $a=0$ as it is easy to check that in this case the (D) condition is violated. 

There are two cases depending on whether $Q$ has the $\xi_3^2$ term or not. Suppose that it does. By completing squares, we obtain
\begin{equation}
\begin{split}
    (P,Q) &\equiv (\xi_1^2\pm \xi_2^2 ,a'\xi_1^2+b'\xi_2^2+c'\xi_1\xi_2+\xi_3^2)
    \\&
    \equiv
    (\xi_1^2 \pm \xi_2^2 ,b''\xi_2^2+c'\xi_1\xi_2+\xi_3^2)=:(P',Q').
\end{split}
\end{equation}
The function $\det{(x_1P'+x_2Q')}$ is given by
\begin{equation}
    2 x_2 \det\begin{bmatrix}
        2x_1 & c'x_2
        \\
        c'x_2 & \pm 2x_1+2b''x_2
    \end{bmatrix}=:x_2D'(x_1,x_2).
\end{equation}
Note that $D'(x_1,x_2)$ does not have a linear factor $x_2$. Therefore, in order for $x_2 D'(x_1, x_2)$ to violate the integrability condition as in \eqref{curvature_zz}, the quadratic form $D'(x_1, x_2)$ must be the square of a linear form. As a consequence, we obtain 
\begin{equation}
    (P,Q) 
    \equiv
    (\xi_1^2 - \xi_2^2 ,b''\xi_2^2+c'\xi_1\xi_2+\xi_3^2)
\end{equation}
and 
\begin{equation}
    (b'')^2 - (c')^2=0.
\end{equation}
By considering either the hyperplane $\xi_1=\xi_2$ or the hyperplane $\xi_1=-\xi_2$, we see that the $(\Rtwo)$ condition fails. 

It remains to handle the case where $Q$ does not have the $\xi_3^2$ term. Notice that $Q$ must contain either $\xi_1 \xi_3$ or $\xi_2\xi_3$. By a change of coordinate in $\xi_1$ and $\xi_2$, we may assume that $Q$ has a $\xi_1\xi_3$ term and 
\begin{equation}
\begin{split}
    (P,Q) &\equiv (F(\xi_1, \xi_2),a\xi_1^2+b\xi_2^2+c\xi_1\xi_2+ \xi_1\xi_3),
\end{split}
\end{equation}
where $F$ is a quadratic form. By taking the hyperplane $\xi_1=0$, we can see that the $(\Rtwo)$ condition is violated. This finishes the proof.

\section{Examples: Proof of Theorem \ref{thm_sharpness_degenerate}}\label{200914section3}

We will show that if a pair $(P, Q)$ is equivalent to $(\xi_1^2\pm \xi_2^2, \xi_3^2)$, $(\xi_1\xi_2+\xi_3^2, \xi_1^2)$ or $(\xi_1 \xi_2, \xi_1\xi_3)$, then for the associated surface $S$, the restriction estimate \eqref{restrictionestimate} fails for $q<4$ or $q\ge 4, p<q/(q-3)$.

\medskip

The case $(\xi_1^2\pm \xi_2^2, \xi_3^2)$ is trivial, as our surface can be written as the ``product" of a parabola and a paraboloid or hyperbolic paraboloid, depending on the choice of the  $\pm$ sign. 

\medskip

Let us next consider the case that $(P, Q)$ is equivalent to $(\xi_1 \xi_2+\xi_3^2, \xi_1^2)$. Let $R$ be a large number. Take the function $f_1=1_{[0,R^{-1}] \times [0,1]\times [0, R^{-1/2}]}$, that is, the characteristic function of the set
\begin{equation}
    \{\xi=(\xi_1, \xi_2, \xi_3)\in [0, 1]^3: 0\le \xi_1\le R^{-1}, 0\le \xi_3\le R^{-1/2}\}.
\end{equation}
Recall the definition of the extension operator $E^S$ in \eqref{0823equ_extension}. We see that if $x$ is chosen such that 
\begin{equation}
|x_1|+|x_4|\le R/100, |x_2|\le 1/100, |x_3|\le R^{1/2}/100, |x_5|\le R^2/100, 
\end{equation}
then $|E^S f_1(x)|\gtrsim R^{-3/2}$. As a consequence, we obtain the constraint $3/q+ 1/p\le 1$. To see the constraint $q\ge 4$, we use Khintchine's inequality. Let $\Box=I_1\times I_2\times I_3$ be a dyadic rectangular box in $[0, 1]^3$ with $|I_1|=R^{-1}, |I_2|=1, |I_3|=R^{-1/2}$. Let $\varepsilon=(\varepsilon_{\Box})$ be a sequence of independent random variables  where $\varepsilon_{\Box}$ takes values $\pm 1$ randomly. We have 
\begin{equation}\label{equ_0823_khintchin}
    \mathbbm{E}_{\varepsilon} \|E^S (\sum_{\Box\subset [0, 1]^3} \varepsilon_{\Box}\cdot  1_{\Box})\|_q^q\simeq \int \Big(\sum_{\Box} |E^S 1_{\Box}|^2\Big)^{q/2}.
\end{equation}
First of all, \eqref{equ_0823_khintchin} is bounded from below by 
\begin{equation}
    \int \sum_{\Box} |E^S 1_{\Box}|^q \gtrsim R^{-\frac{3q}{2}} R^{6}.
\end{equation}
Moreover, in order for \eqref{restrictionestimate} to hold, we need that \eqref{equ_0823_khintchin} is bounded from above by a constant. This leads to the constraint $q\ge 4$. 

\medskip

Regarding the last case $(\xi_1 \xi_2, \xi_1\xi_3)$, essentially the same construction as above works as well. We omit the details.

\section{Interpolations and an epsilon removal lemma}\label{200906section:reduction}

In the following sections, when proving the desired restriction estimate \eqref{restrictionestimate} for the range $q>4$ and $p>q/(q-3)$, we will from time to time apply the trivial estimate 
\begin{equation}
    \|E^S f\|_{\infty}\le \|f\|_1.
\end{equation}
This estimate, together with interpolation, allows us to reduce the above range of $(p, q)$ to the range $p=q>4$.

\medskip

A second reduction, which is known as the epsilon removal lemma in the literature, will be needed in Section \ref{200912section7} to handle surfaces that satisfies the (CM) condition. 

Let $S$ be a $d$-dimensional compact smooth surface in $\R^n$ parametrizd by $\{ (\xi,\mathbf{P}(\xi))\in [0,1]^d \times \R^{n-d} \}$. Let $\psi_0$ be a smooth function such that $\psi_0(\xi)=1$ on $[0,1]^d$ and $\psi_0(\xi)=0$ outside of the cube $[-1,2]^d$. Denote by $E^Sf$ the extension operator associated with the surface $S$ defined by
\begin{equation}
    E^Sf(x',x'')=\int_{[-1,2]^d}e(x' \cdot \xi + x'' \cdot \mathbf{P}(\xi))f(\xi)\psi_0(\xi) \,d\xi.
\end{equation}
The function $\psi_0$ is introduced for some technical reason.

Several versions of epsilon removal lemma have been obtained in the literature \cite{MR1666558, MR2860188, Kim2017SomeRO}. However, it has been mostly stated only for hypersurfaces. We record here the epsilon removal lemma for surfaces with co-dimension larger than one. The idea of the proof is essentially the same as those in the literature.

\begin{thm}
Let $q \geq p \geq 2$. We assume the decay of the function $E^S1$
\begin{equation}\label{fourierdecay}
    |E^S1(x)| \leq C_{S,\rho}(1+|x|)^{-\rho}
\end{equation}
for some $\rho>0$, and assume the local extension estimate
\begin{equation}
    \|E^Sf\|_{L^q(B_R)} \leq C_{S,p,q,\alpha}R^{\alpha}\|f\|_{L^p}
\end{equation}
for every $0<\alpha \ll 1$, $R \geq 1$,  ball $B_R$ of radius $R$, and function $f$. Under these conditions, we have the global extension estimate:
\begin{equation}\label{globalestimate}
    \|E^Sf\|_{L^s(\R^{n})} \leq C_{S,p,s}\|f\|_{L^p}
\end{equation}
for every $s>q$ and function $f$.
\end{thm}

Let us give a sketch of the theorem. One  ingredient is the Calderon-Zygmund decomposition. We will prove the dual estimate of \eqref{globalestimate}
\begin{equation}
    \|\hat{F}\|_{L^{p'}(d\sigma)} \lesssim \|F\|_{L^{s'}(\R^n)},
\end{equation}
which fits better than the extension estimate \eqref{globalestimate} with the Calderon-Zygmund decomposition. Here, $d\sigma$ is the surface measure of the surface $S$. Note that $1<p'<2$.  We will interpolate an $L^1$-estimate and an $L^2$-estimate.
\medskip

We define sparse balls. Let $C$ be a sufficiently large constant depending on the parameter $\rho$ and the dimensions $d,n$. We call a collection of balls $\{B_R(x_i) \}_{i=1}^{N}$ sparse if $|x_i-x_j| \geq (NR)^C$ for any $i \neq j$. We will prove
\begin{equation}\label{sparsereduction}
    \|\hat{F}\|_{L^{p'}(d\sigma) } \leq C_{p,s,\alpha} R^{\alpha} \|F\|_{L^{s'}(\R^n)}
\end{equation}
for any function $F$ supported in the union of the sparse collection of balls $\{B_R(x_i)\}_{i=1}^{N}$ and any $R \geq 1$ and $0<\alpha \ll 1$. Once we obtain this estimate, the rest of the proof is identical to that in \cite{MR1666558} and \cite{MR2860188}. In particular, they do not use the condition that the surface $S$ is a hypersurface, and the same proof works for our surface. Hence, we leave out the details.
\medskip

It remains to prove \eqref{sparsereduction}. We will apply the assumption on the local restriction estimate. By the assumption and H\"{o}lder's inequality, we know that
\begin{equation}
    \|\hat{F}\|_{L^{p'}(\mc{N}_{R^{-1}}(S))} \lesssim R^{-(n-d)/p'}R^{\alpha}\|F\|_{L^{s'} (B_R)}
\end{equation}
for every ball $B_R$, where $\mc{N}_{R^{-1}}(S)$ denotes the $R^{-1}$-neighbourhood of $S$. 
By the support condition on $F$, we may write $F$ as $\sum_{i=1}^{N}f_i(x-x_i)$ where $f_i$ is supported on the ball $B_R(0)$. We take a function $\psi$ such that $|\psi| \simeq 1$ on $B_1(0)$ and $\hat{\psi}$ is supported on $B_1(0)$. Define $\psi_R(x)=\psi(x/R) $ and $h_i(x)=f_i(x)/\psi_R(x)$. The function $h_i$ is well-defined because $\psi_R$ is non-vanishing on the support of $f_i$. We now have
$
    F(x)=\sum_{i=1}^N (h_i \psi_R) (x-x_i)$ and the inequality \eqref{sparsereduction} follows from
    \begin{equation}\label{0908(4.7)}
        \Big\| \sum_{i=1}^N e^{ix_i \cdot \xi}(\widehat{h_i}*\widehat{\psi_R})\Big\|_{L^{p'}_{\xi}(d\sigma) } 
        \lesssim
        R^{\alpha}
        \big(\sum_{i=1}^{N}\|h_i\|_{L^{s'}(\R^n) }^{s'} \big)^{1/s'}.
    \end{equation}
Note that for $\xi \in S$ we have
\begin{equation}
    (\widehat{h_i}*\widehat{\psi_R})(\xi)=
    ((\widehat{h_i} \cdot \chi_{\mc{N}_{R^{-1}}(S) }) *\widehat{\psi_R})(\xi)
\end{equation}
by the condition that the support of $\widehat{\psi_R}$ is contained in $B_{R^{-1}}(0)$. By this observation and the assumption that $s>p$, it is not difficult to see that the desired inequality \eqref{0908(4.7)} follows from
\begin{equation}
    \Big\| \sum_{i=1}^N e^{ix_i \cdot \xi}(g_i*\widehat{\psi_R})\Big\|_{L^{p'}_{\xi}(d\sigma) } 
        \lesssim
        R^{(n-d)/p'}
        \big(\sum_{i=1}^{N}\|g_i\|_{L^{p'}(\R^n) }^{p'} \big)^{1/p'}
\end{equation}
for every function $g_i$.
We will interpolate two estimates at $p'=1$ and $p'=2$.

Let us first show the case that $p'=1$. By the triangle inequality, it suffices to show that
\begin{equation}
    \| g_i*\widehat{\psi_R}  \|_{L^1(d\sigma)} \lesssim R^{n-d} \|g_i\|_{L^1(\R^n)}.
\end{equation}
Note that $n-d$ is the co-dimension of the surface $S$. The above inequality immediately follows from the estimate
\begin{equation}\label{0914}
    \sup_{y \in \R^n}\int_S|\widehat{\psi_R}(\xi-y)|d\sigma(\xi) \lesssim R^{n-d}.
\end{equation}
Let us move on to the case that $p'=2$. We need to prove
\begin{equation}\label{0908(4.11)}
    \Big\| \sum_{i=1}^N e^{ix_i \cdot \xi}(g_i*\widehat{\psi_R})\Big\|_{L^{2}_{\xi}(d\sigma) }^2
        \lesssim
        R^{n-d}
       \sum_{i=1}^{N}\|g_i\|_{L^{2}(\R^n) }^{2}.
\end{equation}
We expand the $L^2$-norm on the left hand side. We have a diagonal term and off-diagonal terms.
\begin{equation}
    \sum_{i=1}^{N}\|
    g_i*\widehat{\psi_R}\|_{L^2(d\sigma)}^2+\sum_{i \neq j}\int
    e^{i(x_i-x_j)\cdot \xi}
    (g_i*\widehat{\psi_R})(\xi)\overline{(g_j*\widehat{\psi_R})(\xi)}d\sigma(\xi).
\end{equation}
The first term can be bounded by the right hand side of \eqref{0908(4.11)} by the estimate \eqref{0914}. The second term also can be bounded by the right hand side of \eqref{0908(4.11)} by following the same proof in page 12 of \cite{Kim2017SomeRO}. This gives the proof of the theorem.
\\

We will apply the epsilon removal lemma to surfaces satisfying the (CM) condition.
Let us check that the surface satisfies the Fourier decay condition \eqref{fourierdecay}. In fact, we will show
\begin{equation}\label{fourierdecaycheck}
    |E^S1(x_1,\ldots,x_5)| \leq C_{S}(1+|x|)^{-1/2}
\end{equation}
for any surface $S$ satisfying the (CM) condition. By rotation, we may assume that $x_5=0$. After this change of variables, the resulting pair $(P',Q')$ is still an irreducible pair. In particular, $P' \not\equiv 0$. We now simply apply van der Corput's lemma and obtain the desired estimate \eqref{fourierdecaycheck}.

\section{Proof of Theorem \ref{0822main_thm}: The (D) and $(\Rzero)$ conditions} \label{200914section5}

In this section, we will prove that for $(P, Q)=(\xi_1 \xi_2, \xi_1\xi_3)$, we have 
\begin{equation}
\|E^S f\|_{L^p(B_R)}\le C_p \|f\|_{L^p},    
\end{equation}
for every $p>4$, $R\ge 1$ and some constant $C_p>0$. The proof relies on the bilinear argument of Christ \cite{MR766216}, see also similar argument in Oberlin \cite{MR2078263}, Bourgain and Demeter \cite{MR3447712} and Oh \cite{MR3848437}. To simplify notation, we use $E f$ instead of $E^S f$, and assume $B_R$ is the ball of radius $R$ centered at the origin. 

\medskip

The proof is via an induction argument on $R$. We assume that we have proven
\begin{equation}\label{0823_induction_dr0}
    \|E f\|_{L^p(B_{R'})}\le C_{p}  \|f\|_p,
\end{equation}
for every $R'\le R/2$, and aim to prove that the same estimate holds with $R$ in place of $R'$.

Let $K\ge 1$ be a large dyadic integer that is to be chosen. Let $\L_K$ denote a dyadic rectangular box of dimension $1/K\times 1\times 1$ that is parallel to the hyperplane $\xi_1=0$. Let $\mathfrak{P}(K)$ be the partition of $[0, 1]^3$ into dyadic rectangular boxes $\L_K$. Let $\nu>0$. We say that two sets $\Xi, \Xi'\subset [0, 1]^3$ are $\nu$-transverse if for every $\xi=(\xi_1, \xi_2, \xi_3)\in \Xi$ and $\xi'=(\xi'_1, \xi'_2, \xi'_3)\in \Xi'$, it holds that $|\xi_1-\xi'_1|\ge \nu$. Write 
\begin{equation}
    E_{[0, 1]^3} f(x)=\sum_{\alpha\in \mathfrak{P}(K)} E_{\alpha} f(x), \ \ x\in \R^5. 
\end{equation}
To proceed, we run a simple version of the  broad-narrow analysis from Bourgain and Guth \cite{MR2860188}. For each $x \in \R^5$, we define the significant part
\begin{equation}
    \mathcal{C}(x):=\Big\{\alpha \in \mathfrak{P}(K):|E_{[0,1]^3}f(x)| <10K |E_{\alpha}f(x)|  \Big\}.
\end{equation} 
By considering two possible cases $|\mathcal{C}(x)| \geq 3$ and $|\mathcal{C}(x)| \leq 2$,
we  obtain the following simple pointwise estimate: 
\begin{equation}
    \begin{split}
        & |E_{[0, 1]^3}f(x)| \le C\Big(\sum_{\L_{K/2}}|E_{\L_{K/2}} f(x)|^p \Big)^{1/p}\\
        & + CK^{10} \sum_{\alpha_1, \alpha_2: K^{-1}-\text{transverse}} \prod_{\iota=1}^2 |E_{\alpha_{\iota}}f(x)|^{1/2}.
    \end{split}
\end{equation}
Here $C$ is an absolute constant whose value may vary from line to line, and
the sum in the last term is over rectangular boxes $\L_K$. 
As a consequence, we obtain 
\begin{equation}\label{0824_e5.4}
\begin{split}
    \|E_{[0, 1]^3}f\|_{L^p(B_R)}& \le CK^{10} \sum_{\alpha_1, \alpha_2: K^{-1}-\text{transverse}} \Big\|\prod_{\iota=1}^2 |E_{\alpha_{\iota}} f|^{1/2} \Big\|_{L^p(B_R)}\\
    & + C\Big(\sum_{\L_{K/2}\subset [0, 1]^3} \|E_{\L_{K/2}}f\|_{L^p(B_R)}^p \Big)^{1/p}.
\end{split}
\end{equation}
The latter term will be handled via an anisotropic scaling and induction. Let us first take $\L_{K/2}=[0, 2/K]\times [0, 1]\times [0, 1]$. Write
\begin{equation}\label{200905e6.5}
    E_{\L_{K/2}} f(x)=\int_{0}^{2/K}\int_0^1 \int_0^1 f(\xi) e(\xi\cdot x'+ \xi_1\xi_2 x_4+\xi_1 \xi_3 x_5)d\xi,
\end{equation}
where $x'=(x_1, x_2, x_3)$. We apply the change of coordinates \begin{equation}\label{200905e6.6}
    \xi_1\to 2\xi_1/K, \ \xi_2\to \xi_2, \ \xi_3\to \xi_3,
\end{equation}
and obtain 
\begin{equation}
    \|E_{\L_{K/2}} f\|_{L^p(B_R)}\le CK^{-1} K^{3/p} \|E \widetilde{f}\|_{L^p(B_R)},
\end{equation}
where $\widetilde{f}(\xi_1, \xi_2, \xi_3):=f(2\xi_1/K, \xi_2, \xi_3)$. To
apply the induction hypothesis \eqref{0823_induction_dr0}, we break the ball $B_R$ into smaller balls of radius $R/2$ and obtain 
\begin{equation}
\|E_{\L_{K/2}} f\|_{L^p(B_R)}\le C C_p K^{-1} K^{4/p}\|f_{\L_{K/2}}\|_p,
\end{equation}
where $C$ is an absolute constant. The same estimate also holds for other $\L_{K/2}$. As a consequence, we see that the second term on the right hand side of \eqref{0824_e5.4} can be bounded by 
\begin{equation}\label{200905e6.8}
    C C_p K^{-1} K^{4/p}\Big(\sum_{\L_{K/2}}\|f_{\L_{K/2}}\|_p^p\Big)^{1/p}.
\end{equation}
Recall that $p>4$. We will pick $K$ such that $C K^{-1+4/p}=1/2$.

It remains to handle the first term in \eqref{0824_e5.4}. We will actually prove a stronger estimate: 
\begin{equation}\label{0824_plancherel}
    \Big\||E_{\alpha_1} f|^{1/2} |E_{\alpha_2} f|^{1/2}\Big\|_{L^4(\R^5)} \le C_K \|f\|_2,
\end{equation}
for some constant $C_K$ that is allowed to depend on $K$. We write the $4$-th power of the left hand side of \eqref{0824_plancherel} as 
\begin{equation}\label{0823equ5.10}
    \int_{\R^5} \Big|\iint_{\alpha_1\times \alpha_2} f(\xi)f(\xi') e(x'\cdot (\xi+\xi')+ (\xi_1\xi_2+\xi'_1\xi'_2)x_4+(\xi_1\xi_3+\xi'_1\xi'_3)x_5)d\xi d\xi' \Big|^2 dx,
\end{equation}
where $\xi=(\xi_1, \xi_2, \xi_3)$ and $\xi'=(\xi'_1, \xi'_2. \xi'_3)$. The idea of the following change of variables is classical, and has been used vastly in the literature, for instance, Carleson and Sjolin \cite{MR361607}, Fefferman \cite{MR257819}, and Zygmund \cite{MR387950}. Apply the change of variables 
\begin{equation}
    \begin{split}
        & \xi+\xi'=\eta;\ \ \xi_1=\eta_6;\\ 
        & \xi_1\xi_2+\xi'_1\xi'_2=\eta_4;\\
        & \xi_1\xi_3+\xi'_1\xi'_3=\eta_5,
    \end{split}
\end{equation}
where $\eta=(\eta_1, \eta_2, \eta_3)$. The Jacobian of the change of variable is $J:=|\xi_1-\xi'_1|^2$. The expression in \eqref{0823equ5.10} becomes 
\begin{equation}\label{0823equ5_12}
    \int_{\R^5} \Big|\iint F(\eta) e(x\cdot \eta')d\eta' d\eta_6 \Big|^2 dx,
\end{equation}
where $\eta'=(\eta_1, \dots, \eta_5)$ and $\eta=(\eta', \eta_6)$ for some function $F$. We apply H\"older's inequality to $d\eta_6$ and Plancherel's theorem to the variables $\eta'$, and bound \eqref{0823equ5_12} by  
\begin{equation}
    \lesim_K \int |F(\eta)|^2 d\eta\le C_K \|f\|_2^4
\end{equation}
for some $C_K$ that depends only on $K$. 
In the last step, we have changed variables back. This finishes the proof of \eqref{0824_plancherel}. By interpolating with the trivial $L^1 \rightarrow L^{\infty}$ estimates, we obtain 
\begin{equation}\label{200906e5.16}
    \Norm{\prod_{\iota=1}^2 |E_{\alpha_{\iota}} f|^{1/2}}_{L^p(B_R)}\le C_{p, K} \|f\|_p,
\end{equation}
for every $p\ge 4$. 

\medskip

Now we apply \eqref{200905e6.8} and \eqref{200906e5.16} to \eqref{0824_e5.4} and obtain 
\begin{equation}
    \|E f\|_{L^p(B_R)}\le C_p/2 \|f\|_p+ C_{p, K} K^{20}\|f\|_p.
\end{equation}
In the end, we notice that we just need to pick $C_p$ in \eqref{0823_induction_dr0} to be $C_{p, K} K^{200}$ and we can close the induction.

\section{Proof of Theorem \ref{0822main_thm}: The (D) and $(\Rone)$ conditions}\label{200914section6}

In this section, we will prove that if $(P, Q)$ is equivalent to $(\xi_1\xi_2, \xi_1\xi_3+\xi_2^2)$ or $(\xi_1\xi_2, \xi_1^2\pm\xi_3^2)$, then 
\begin{equation}
    \|E^S f\|_{L^p(B_R)}\le C_p \|f\|_{L^p},
\end{equation}
for every $p>4$, and some constant $C_p$ depending only on $p$. The proofs for these two cases are almost the same, and therefore we will only present that of $(\xi_1\xi_2, \xi_1\xi_3+\xi_2^2)$. 

The proof is via an induction on $R$, and we assume that we have proven 
\begin{equation}
    \|E f\|_{L^p(B_{R'})}\le C_p \|f\|_p,
\end{equation}
for every $R'\le R/2$. To proceed, we follow the broad-narrow analysis of the previous section, and see that \eqref{0824_e5.4} holds as well for the current surface. The bilinear term can again be handled in the same way as in \eqref{0824_plancherel}. So far we have proved that 
\begin{equation}\label{0824_e5.4zz}
\begin{split}
    \|E_{[0, 1]^3}f\|_{L^p(B_R)}& \le C_{p, K} \|f\|_p + C\Big(\sum_{\L_{K/2}\subset [0, 1]^3} \|E_{\L_{K/2}}f\|_{L^p(B_R)}^p \Big)^{1/p},
\end{split}
\end{equation}
for every $p>4$. The difference lies in how to handle to last term. Our new surface possesses a different scaling structure. Let us use the case $\L_{K/2}=[0, 2/K]\times [0, 1]^2$ to explain the argument. The other terms can be handled in more or less the same way. We will show that 
\begin{equation}\label{200905e7.3}
    \|E_{\L_{K/2}}f\|_{L^p(B_R)} \lesim_{p, \epsilon} K^{\frac12(\frac12-\frac1p)+\epsilon} \Big(\sum_{\ell(J)=K^{-1/2}} \|E_{[0, 2/K]\times J\times [0, 1]} f\|_{L^p(w_{B_R})}^p\Big)^{1/p}, 
\end{equation}
for every $6\ge p\ge 4$ and $\epsilon>0$, where $w_{B_R}(x):=(1+|x|/R)^{-500}$ and the sum is over dyadic intervals $J$ of side length $K^{-1/2}$. To see \eqref{200905e7.3}, we notice that on $\L_{K/2}$ we have $|\xi_1|\le 2/K$ and by uncertainty principle, the surface $(\xi_1, \xi_2, \xi_3, \xi_1\xi_2, \xi_1 \xi_3+\xi_2^2)$ behaves the same as $(\xi_1, \xi_2, \xi_3, 0, \xi_2^2)$ on $B_K\subset B_R$. Therefore \eqref{200905e7.3} follows from the decoupling inequality of Bourgain and Demeter \cite{MR3374964} for the parabola. We apply \eqref{200905e7.3} to \eqref{0824_e5.4zz}, and obtain 
\begin{equation}
    \begin{split}
        & \|E_{[0, 1]^3} f\|_{L^p(B_R)}\le C_{p, K}\|f\|_p\\
        & + C_{p, \epsilon} K^{\frac12(\frac12-\frac1p)+\epsilon} \Big(\sum_{\ell(J)=K^{-1/2}} \|E_{[0, 2/K]\times J\times [0, 1]} f\|_{L^p(w_{B_R})}^p\Big)^{1/p}.
    \end{split}
\end{equation}
The rest is the standard induction argument, essentially the same as in \eqref{200905e6.5}--\eqref{200905e6.8}, with the only difference of replacing \eqref{200905e6.6} by 
\begin{equation}
    \xi_1\mapsto \xi_1/K, \xi_2\mapsto \xi_2/\sqrt{K}, \xi_3\mapsto \xi_3.
\end{equation}
We leave out the details.

\section{Proof of Theorem \ref{0822main_thm}: The (CM) condition}\label{200912section7}

In this section, we will prove Theorem \ref{0822main_thm} for every pair $(P, Q)$ that satisfies the (CM) condition. By the reduction in Section \ref{200906section:reduction}, it suffices to prove that 
\begin{equation}\label{0823_cm_restriction}
    \|E^S f\|_{L^p(B_R)} \le C_{p, \epsilon} R^{\epsilon} \|f\|_{L^p},
\end{equation}
for every $p>4, \epsilon>0$, ball $B_R\subset \R^5$ of radius $R$, surface $S$ given by \eqref{0822_surface} with $(P, Q)$ satisfying the (CM) condition. 
As $S$ is fixed throughout this section, we will abbreviate $E^S$ to $E$. 

\medskip

It turns out that there are several ways of proving \eqref{0823_cm_restriction}. Here we will present a proof that has the potential of being further generalized and improved. In this approach, we will use the broad-narrow analysis of \cite{MR2860188}, the linear algebra computations in \cite{MR3945730} that allow us to handle the broad case via Brascamp-Lieb inequalities of \cite{MR3300318} and \cite{MR3783217}, and the bootstrapping argument of \cite{arxiv:1902.03450} that allows us to handle the narrow case.

\subsection{Transversality and multi-linear estimates} Let $d=3$. Let $(V_j)_{j=1}^M$ be a tuple of linear subspaces $V_j\subset 
\R^{d+2}$ of dimension $d$. Let $\pi_j: \R^{d+2}\rightarrow V_j$ denote the orthogonal projection onto $V_j$. We define the Brascamp-Lieb constant $BL((V_j)_{j=1}^M)$ to be the smallest constant $BL$ such that the inequality 
\begin{equation}
    \int_{\R^{d+2}} \prod_{j=1}^M f_j(\pi_j(x))^{\frac{d+2}{dM}}\le (BL)\prod_{j=1}^M \Big(\int_{V_j} f_j(y)dy\Big)^{\frac{d+2}{dM}}, 
\end{equation}
holds for all non-negative functions $f_j: V_j\to \R$. 
\begin{defi}
Let $\nu>0$. A tuple of subsets $R_1, \dots, R_M\subset [0, 1]^d$ is called $\nu$-transverse if 
\begin{equation}
    BL((V(t_j))_{j=1}^M)\le \nu^{-1},
\end{equation}
for every $t_j\in R_j$, where $V(t)$ denotes the tangent space of the surface $S$ at $t\in \R^d$. 
\end{defi}
Bennett, Carbery, Christ and Tao \cite{MR2377493} provided a complete characterization for the finiteness of various Brascamp-Lieb constants, including the one above. To use their characterization, we need the following lemma. 
\begin{lem}[\cite{MR3945730}, \cite{arxiv:1902.03450}]\label{200906lem4.2} Let $S$ be a surface that satisfies the (CM) condition. 
For every subspace $V\subset \R^{d+2}$, one of the following holds: 
\begin{itemize}
    \item[(1)] $\dim \pi_t(V)\ge \frac{d}{d+2} \dim V$ for every $t\in \R^d$, or 
    \item[(2)] $\dim \pi_t(V)>\frac{d}{d+2} \dim V$ for some $t\in \R^d$.
\end{itemize}
Here $\pi_t$ denotes the orthogonal projection onto $V(t)$. 
\end{lem}
This lemma was implicit in \cite{MR3945730}. For a proof, see Subsection 3.2 of \cite{arxiv:1902.03450}. In \cite{arxiv:1902.03450}, Lemma \ref{200906lem4.2} was formulated as Hypothesis 2.26, and verified in Subsection 3.2 under the assumption that the relevant surface satisfies the (D) condition. As our surface $S$ satisfies the (CM) condition, which is stronger than the (D) condition, we also have Lemma \ref{200906lem4.2} in the current situation.

As a consequence of Lemma \ref{200906lem4.2}, we have the following lemma, which provides a criteria  for checking whether a collection of subsets of $[0, 1]^d$ is transverse or not. Let $K>0$ be a large dyadic number. Let $\mc{P}(1/K)$ be the partition of $[0, 1]^d$ into dyadic cubes of side length $1/K$. 

\begin{lem}[Lemma 2.27 \cite{arxiv:1902.03450}]\label{200906lem7.3}
There exists $\theta>0$ such that the following holds: For every $K$, there exists $\nu_K>0$ such that for every collection $\mc{R}\subset \mc{P}(1/K)$, one of the following alternatives holds:
\begin{itemize}
    \item[(1)] $\mc{R}$ is $\nu_K$-transverse, or 
    \item[(2)] There exists a sub-variety $\mc{Z}$ of degree at most $d$ such that 
    \begin{equation}
        |\{R\in \mc{R}: 2R\cap \mc{Z}\neq \emptyset\}|> \theta |\mc{R}|. 
    \end{equation}
\end{itemize}
\end{lem}
We should also mention that analogues of this lemma already appeared in \cite{MR3614930}, \cite{MR3709122}, \cite{MR3994585} and \cite{arxiv:1811.02207}. 

\begin{lem}[\cite{MR3300318}, Theorem 1.3 \cite{MR3783217}]\label{200906lem7.4}
Let $\{R_j\}_{j=1}^M$ be a collection of cubes in $\mc{P}(1/K)$ that are $\nu_K$-transverse. Then 
\begin{equation}
    \Big\|\prod_{j=1}^M|E_{R_j} f|^{1/M}\Big\|_{L^p(B_R)}\lesim_{K, p, \epsilon'} R^{\epsilon'} \|f\|_2, 
\end{equation}
for every $p>10/3, \epsilon'>0$ and ball $B_R\subset \R^{d+2}$ with $d=3$. 
\end{lem}
This lemma suggests that we may expect \eqref{restrictionestimate} to hold for $q>10/3$, whenever $S$ satisfies the (CM) condition.

\subsection{Decoupling constants and lower dimensional contributions} Let $d_1\ge 2$ be an integer. Let $P_1(\xi), Q_1(\xi)$ be two quadratic forms depending on $\xi\in \R^{d_1}$. Let $S_1$ be the surface $\{(\xi, P_1(\xi), Q_1(\xi)): \xi\in [0, 1]^{d_1}\}$. For $p\ge 2$ and dyadic $\delta\in (0, 1)$, let $\dec_{S_1}(\delta, p)$ be the smallest constant $\dec$ such that 
\begin{equation}
    \|E^{S_1}_{[0, 1]^{d_1}} f\|_{L^{p}(\R^{d_1+2})}\le (\dec) \Big(\sum_{\ell(\Box)=\delta}\|E^{S_1}_{\Box} f\|_{L^p(\R^{d_1+2})}^p\Big)^{1/p},
\end{equation}
for every function $f$, where the sum on the right hand side runs through all dyadic cubes of side length $\delta$. 

Let us turn to our surface $S$. Recall that $d=3$. Let $H$ be a hyperplane in $\R^d$ that passes through the origin. By restricting $S$ to $H$, we obtain a two-dimensional quadratic surface in $\R^4$. Let us use $S|_H$ to denote this surface. Let $\Lambda>0$ be such that 
\begin{equation}\label{200906e4.6}
    \dec_{S|_H}(\delta, p)\lesim_{\Lambda, p} \delta^{-\Lambda}, \text{ for every } \delta\in (0, 1), \text{ uniformly in } H.  
\end{equation}
The following lemma allows us to handle the narrow case in the forthcoming broad-narrow analysis of Bourgain and Guth \cite{MR2860188}. 

\begin{lem}[Corollary 2.18 \cite{arxiv:1902.03450}]\label{200906lem7.5}
For every $D_0\ge 1$ and $\epsilon'>0$, there exists $c=c(D_0, \epsilon')>0$ such that, for every sufficiently large $K$, there exist
\begin{equation}\label{200906e7.8}
    K^c\le K_1\le \dots \le K_{D_0}\le \sqrt{K}
\end{equation}
such that, for every non-zero polynomial $P_{D_0}$ of degree $D_0$, there exist collections of pairwise disjoint cubes $\mc{G}_j\subset \mc{P}(1/K_j), j=1, 2, \dots, D_0$, such that
\begin{equation}
    \mc{N}_{1/K}(\mc{Z}(P_{D_0}))\cap [0, 1]^d\subset \bigcup_{j=1}^{D_0}\bigcup_{\Box\in \mc{G}_j} \Box,
\end{equation}
and 
\begin{equation}\label{200906e7.10}
    \|\sum_{\Box\in \mc{G}_j} f_{\Box}\|_p \lesim_{D_0, \epsilon'} K_j^{\Lambda+\epsilon'}\Big(\sum_{\Box\in \mc{G}_j} \|f_{\Box}\|_p^p \Big)^{1/p}.
\end{equation}
Here $\mc{Z}(P_{D_0})$ denotes the zero set of $P_{D_0}$, and $\mc{N}_{1/K}(\mc{Z})$ denotes the $1/K$-neighbourhood of $\mc{Z}$. 
\end{lem}

Let us explain how this lemma will be applied later. Recall that our surface $S$ satisfies the (CM) condition, and therefore it satisfies the $(R_2)$ condition. By the definition of the $(R_2)$ condition in \eqref{equ_rank_condition}, we know that 
\begin{equation}
     \min_{\text{ hyperplane } H\subset \R^3} \max_{\lambda_1^2+ \lambda_2^2=1} \rank\Big( (\lambda_1 P(\cdot)+\lambda_2 Q(\cdot))|_{H}\Big)=2.
\end{equation}
Via a simple compactness argument, we see that for every hyperplane $H\subset \R^3$, there exist $\lambda_1$ and $\lambda_2$ with $\lambda_1^2+\lambda_2^2=1$ such that 
\begin{equation}
    \det
    \Big(\mathrm{Hessian}\Big( (\lambda_1 P(\cdot)+\lambda_2 Q(\cdot))|_{H}\Big)\Big)\simeq 1,
\end{equation}
where the implicit constant in the last estimate depends only on $P$ and $Q$, and in particular, does not depend on $H$. This, combined with the 
decoupling inequalities for paraboloids and hyperbolic paraboloids due to Bourgain and Demeter in  \cite{MR3374964} and \cite{MR3736493}, implies that   \eqref{200906e4.6} holds with $\Lambda=2(1-3/p)+\epsilon'$ for every $p\ge 4$ and every $\epsilon'>0$.

\subsection{Broad-narrow analysis of Bourgain and Guth.}  
In this subsection, we run the Bourgain-Guth argument in \cite{MR2860188} to finish the proof of \eqref{0823_cm_restriction}. The proof is via an induction on $R$, and therefore we assume that we have proven 
\begin{equation}
    \|E f\|_{L^p(B_{R'})}\le C_{p, \epsilon} (R')^{\epsilon} \|f\|_{L^p}, 
\end{equation}
for every $R'\le R/2$. Let $K$ be a large dyadic number that is to be chosen. Write 
\begin{equation}
    E_{[0, 1]^d}f(x)=\sum_{\alpha\in \mc{P}(1/K)} E_{\alpha} f(x).
\end{equation}
We essentially follow the proof of Proposition 2.35 \cite{arxiv:1902.03450}. Fix $B'=B_K\subset B_R$, a ball of radius $K$. By the uncertainty principle, $|E_{\alpha}f|$ is essentially a constant on $B'$. We use $|E_{\alpha}f|(B')$ to denote this constant. Initialize 
\begin{equation}
    \S_0:=\{\alpha\in \P(K^{-1}): |E_{\alpha}f|(B')\ge K^{-d} \max_{\alpha'\in \P(K^{-1})} |E_{\alpha'} f|(B')\}.
\end{equation}
We repeat the following algorithm. Let $w\ge 0$. If $\S_w=\emptyset$ or $\S_w$ is $\nu_K$-transverse, then we set $\mc{T}:=\S_{w}$ and terminate. Otherwise, by Lemma \ref{200906lem7.3}, there exists a sub-variety $\mc{Z}$ of degree at most $d$ such that 
\begin{equation}
    |\{\alpha\in \S_w: 2\alpha\cap \mc{Z}\neq \emptyset\}|\ge \theta |\S_w|.
\end{equation}
Let $\mc{G}_{w, j}:=\mc{G}_j$ be given by Lemma \ref{200906lem7.5}. Repeat the algorithm with 
\begin{equation}
    \S_{w+1}:=\S_w\setminus \bigcup_{j=1}^3 \bigcup_{\Box\in \mc{G}_{w, j}} \P(\Box, 1/K).
\end{equation}
This algorithm terminates after $O(\log K)$ steps. 

Now we process the cubes in $\mc{G}_{w, j}$ and in $\mc{T}$. To avoid multiple counting, we define 
\begin{equation}
\widetilde{\mc{G}}_{w, j}:=\Big( \mc{G}_{w, j}\setminus \bigcup_{0\le w'<w} \mc{G}_{w', j}\Big)\setminus \bigcup_{1\le j'<j}\bigcup_{w'} \bigcup_{\Box\in \mc{G}_{w', j'}} \P(\Box, 1/K_j).
\end{equation}
Under the above notation, we obtain 
\begin{equation}\label{200906e7.17}
    \begin{split}
        & \|E f\|_{L^p(B')} \le C\Big(\sum_{\alpha\in \P(1/K)} \|E_{\alpha} f\|_{L^p(B')}^p \Big)^{1/p}\\
        & + C\sum_{w\lesim \log K} \sum_{j=1}^3 \Big\| \sum_{\beta\in \widetilde{\mc{G}}_{w, j}} E_{\beta} f\Big\|_{L^p(B')}+ K^{100}  \Big\|\prod_{\alpha\in \mc{T}} |E_{\alpha} f|^{\frac{1}{|\mc{T}|}}\Big\|_{L^p(B')}.
    \end{split}
\end{equation}
By a simple localization argument and the estimate \eqref{200906e7.10}, we see that 
\begin{equation}\label{200906e7.18}
    \Big\| \sum_{\beta\in \widetilde{\mc{G}}_{w, j}} E_{\beta} f\Big\|_{L^p(B')} \lesim_{\epsilon', p} K_j^{2(1-3/p)+\epsilon'} \Big(\sum_{\beta}\|E_{\beta} f\|_{L^p(w_{B'})}^p \Big)^{1/p},
\end{equation}
for every $\epsilon'>0$ and $p\ge 4$, where $w_{B'}(x):=(1+|x-c_{B'}|/r_{B'})^{-500}$ with $c_{B'}$ the center of $B'$ and $r_{B'}$ it radius. Now we apply \eqref{200906e7.18} to \eqref{200906e7.17}, sum over all balls $B'\subset B_R$, and obtain 
\begin{equation}\label{200906e7.19}
    \begin{split}
        & \|E f\|_{L^p(B_R)} \le C\Big(\sum_{\alpha\in \P(1/K)} \|E_{\alpha} f\|_{L^p(B_R)}^p \Big)^{1/p}\\
        & +  \sum_{j=1}^3 C_{p, \epsilon'} K_j^{2(1-3/p)+\epsilon'} \Big(\sum_{\beta\in \P(1/K_j)}\|E_{\beta} f\|_{L^p(w_{B_R})}^p \Big)^{1/p}\\
        & + K^{100} \sum_{1\le M\le K^3} \sum_{\substack{\alpha_1, \dots, \alpha_M\in \P(1/K)\\ \nu_K-\text{transverse}}}  \Big\|\prod_{j=1}^M  |E_{\alpha_j} f|^{\frac{1}{M}}\Big\|_{L^p(B_R)}.
    \end{split}
\end{equation}
Here we have used \eqref{200906e7.8} and absorbed $\log K$ by $K_j^{\epsilon'}$. At this moment, it will be clear how we proceed: We apply induction to the first and second terms, in the same way as in \eqref{200905e6.5}--\eqref{200905e6.8}, and the multi-linear restriction estimates in Lemma \ref{200906lem7.4} to the last term. The details are left out.

\medskip

\section{Proof of Theorem \ref{0822main_thm}: The last case}\label{200914section8}

In this section, we prove Theorem \ref{0822main_thm} for the case where the (D) condition fails. According to Theorem \ref{thm_classification_of_surfaces}, we need to consider two surfaces $(\xi_1, \xi_2, \xi_3, \xi_1^2\pm \xi_2^2, \xi_3^2)$ and $(\xi_1, \xi_2, \xi_3, \xi_1 \xi_2+\xi_3^2, \xi_1^2)$. The desired restriction estimate \eqref{restrictionestimate} for the former surface at $p=q>4$ follows from well-known restriction estimates for the parabola $(\xi_3, \xi_3^2)$ and the (hyperbolic) paraboloid $(\xi_1, \xi_2, \xi_1^2\pm \xi_2^2)$. Therefore we only need to consider the latter surface. Throughout this section, we use $S$ to denote the surface under consideration.  The proof of \eqref{restrictionestimate} for such a surface contains the main novelty of the paper. We continue to use $Ef$ to stand for $E^{S}f$ for the sake of simplicity.

\medskip

Let us start with defining several constants.

\begin{defi}
Let $p\ge 2$ be given. Let $D_1(R, p)$ be the smallest constant such that
\begin{equation}\label{200901e7.1}
    \|E_{[0,1]^3}f\|_{L^p(B_R)} \leq D_1(R, p)\|f\|_{L^p(\R^3)},
\end{equation}
for every $f: [0,1]^3 \rightarrow \C$ and ball $B_R\subset \R^5$ of radius $R$.
Define $D_2(M,N,R, p)$ to be the smallest constant such that
\begin{equation}\label{200901e7.2}
    \|E_{I \times [0,1] \times J }f\|_{L^p(B_R)} \leq D_2(M, N, R, p)\|f\|_{L^p(\R^3)}
\end{equation}
for every $f:[0,1]^3 \rightarrow \C$, ball $B_R\subset \R^5$ and intervals $I, J\subset [0, 1]$ with side length $M^{-1}$ and $N^{-1}$, respectively. 
\end{defi}

By interpolation with a trivial $L^1\rightarrow L^{\infty}$ bound, what we need to prove becomes 
\begin{equation}\label{d1bound}
    D_1(R, p) \leq C_{p}
\end{equation}
for every $p>4$ and $R\ge 1$. As $p>4$ will be fixed throughout this section, we will suppress the dependence of $D_1$ and $D_2$ on it.  Note that
\begin{equation}\label{domination}
    D_1(R)=D_2(1,1,R)
\end{equation}
and
\begin{equation}\label{domination2}
    D_2(M_1,N_1,R) \leq D_2(M_2,N_2,R)
\end{equation}
whenever $M_1 \geq M_2$ and $N_1 \geq N_2$.

\subsection{Useful estimates. }
As intermediate steps of proving \eqref{d1bound}, we need to appeal to restriction estimates for certain lower dimensional surfaces. If we restrict our surface $S$  to the $\xi_2\xi_3$-plane, then our surface becomes $\{(\xi_2,\xi_3,\xi_3^2)\}$. In this case, we will apply well known restriction estimates for the parabola. If we restrict $S$ to the $\xi_1\xi_2$-plane, then the surface becomes $S_6:=\{(\xi_1,\xi_2,\xi_1^2,\xi_1\xi_2) \}$. The following proposition records a sharp restriction estimate for such a surface. 

\begin{prop}[A restriction estimate for the surface $S_6$]\label{lowerdim}
For $q_1>5$ and $1/p_1+4/q_1<1$, it holds that
\begin{equation}\label{200831e7.6}
    \|E^{S_6}f\|_{L^{q_1}(\R^4 )} \leq C_{p_1,q_1}\|f\|_{L^{p_1}(\R^2)}.
\end{equation}
Moreover, the range for $p_1$ and $q_1$ is sharp, up to endpoints. 
\end{prop}

Later we will use $L^p$ bounds of $E^{S_6} f$ for $p>4$. To obtain such estimates, we will interpolate the restriction estimate in Proposition  \ref{lowerdim} with $p_1=q_1>5$ and a trivial $L^2$-estimate and 
obtain
\begin{equation}\label{200901e7.7}
    \|E_{[0,1]^2}^{S_6}f\|_{L^p(B_R)} \lesssim_{p} R^{\frac16}\|f\|_{L^p(\R^2)}
\end{equation}
for every $p>4$.

Sharp restriction estimates for  two-dimensional  surfaces in $\R^4$ that satisfy certain non-degenerate curvature conditions were obtained by Christ \cite{MR766216}. The surface $S_6$ does not satisfy Christ's non-degeneracy condition, and therefore restriction estimates of the form \eqref{200831e7.6} hold true for a smaller range of $q_1$, which is $q_1>5$, than the range $q_1>4$ for non-degenerate surfaces. 

The proof of Proposition \ref{lowerdim} is rather straightforward. As we could not find a reference in the literature, we will provide a proof here.

\begin{proof}[Proof of Proposition \ref{lowerdim}]
To simplify notation, $E^{S_6}$ will be simplified to $E'$. By interpolation with the trivial $L^1\to L^{\infty}$, it suffices to prove 
\begin{equation}
    \|E' f\|_{L^{p_1}(B_R)}\le C_{p_1} \|f\|_{p_1},
\end{equation}
for every $p_1>5$ and ball $B_R\subset \R^4$ of radius $R$. We may assume that $B_R$ is centered at the origin. The proof is via an induction on $R$. Assume that we have proven  
\begin{equation}\label{200905e8.9}
    \|E' f\|_{L^{p_1}(B_{R'})}\le C_{p_1} \|f\|_{p_1},
\end{equation}
for every $R'\le R/2$. Let $K$ be a large positive dyadic number that is to be chosen. Write 
\begin{equation}
    E'_{[0, 1]^2} f=\sum_{I\subset [0, 1]; |I|=1/K} E'_{I\times [0, 1]} f. 
\end{equation}
Here the sum is over dyadic intervals. For two dyadic intervals $I_1, I_2$ of length $1/K$, we say that they are disjoint if their distance is at least $1/K$. Observe that we have the following simple pointwise estimate 
\begin{equation}
    |E'_{[0, 1]^2} f|\le 2 \sup_{|I|=1/K} |E'_{I\times [0, 1]} f|+ K^{10}\sup_{I_1, I_2 \, \text{disjoint}} |E'_{I_1\times [0, 1]}f E'_{I_2\times [0, 1]}f|^{1/2}.
\end{equation}
As a consequence, we obtain 
\begin{equation}
\begin{split}
    & \|E'_{[0, 1]^2} f\|_{L^{p_1}(B_R)}  \le C \Big(\sum_{|I|=1/K} \|E'_{I\times [0, 1]} f\|_{L^{p_1}(B_R)}^{p_1} \Big)^{1/p_1}\\
    & + K^{10}\sum_{I_1, I_2 \text{disjoint}}\Big( \Big\||E'_{I_1\times [0, 1]}f E'_{I_2\times [0, 1]}f|^{1/2}\Big\|_{L^{p_1}(B_R)}^{p_1}\Big)^{1/p_1}.
    \end{split}
\end{equation}
Regarding the bilinear term, we even have a good $L^4$ estimate. To be more precise, we have \begin{equation}\label{200905e8.14}
    \Big\||E'_{I_1\times [0, 1]}f E'_{I_2\times [0, 1]}f|^{1/2}\Big\|_{L^{4}(B_R)} \le C_K \|f\|_{L^4},
\end{equation}
via Plancherel's theorem. This was already observed by Bourgain and Demeter \cite{MR3447712}. So far we have obtained 
\begin{equation}
    \|E'_{[0, 1]^2} f\|_{L^{p_1}(B_R)}\le C \Big(\sum_{|I|=1/K} \|E'_{I\times [0, 1]} f\|_{L^{p_1}(B_R)}^{p_1} \Big)^{1/p_1}+ C_K K^{12}\|f\|_{p_1}.  
\end{equation}
Regarding the former term, we apply the induction hypothesis \eqref{200905e8.9}. Take the case $I=[0, 1/K]$ as an example. Denote $\widetilde{f}(\xi_1, \xi_2)=f(\xi_1/K, \xi_2)$.
We apply the change of variables 
\begin{equation}
    \xi_1\mapsto \xi_1/K,
     \ \  \xi_2\mapsto \xi_2,
\end{equation}
and see that 
\begin{equation}
    \|E'_{I\times [0, 1]}f\|_{L^{p_1}(B_R)} \le K^{-1} K^{4/p_1} \|E' \widetilde{f}\|_{L^{p_1}(B_R)}.
\end{equation}
By cutting $B_R$ into balls of radius $R/2$ and applying the induction hypothesis, we further obtain 
\begin{equation}
    \|E'_{I\times [0, 1]}f\|_{L^{p_1}(B_R)} \le 4 C_{p_1} K^{-1} K^{4/p_1} K^{1/p_1} \|f_{I\times [0, 1]}\|_{L^{p_1}}.
    \end{equation}
    We apply this bound to \eqref{200905e8.14} and obtain 
    \begin{equation}
        \|E'_{[0, 1]^2}f\|_{L^{p_1}(B_R)}\le 8 C_{p_1} K^{-1+5/p_1}\|f\|_{p_1}+ C_K K^{12}\|f\|_{p_1}. 
    \end{equation}
    In the end, we close the induction by first picking $K$ so large that $8K^{-1+5/p_1}=1/2$ and then setting $C_{p_1}=2C_K K^{12}$. 
    
    The sharpness of the exponents can be seen by using the same example as in Theorem \ref{thm_sharpness_degenerate}. We leave out the detailed constructions. 
\end{proof}

\medskip

In order to prove \eqref{d1bound}, we use an inductive argument on the radius $R$. We may assume that $R$ is a large number by taking $C_{p}$ to be sufficiently large. 
The restriction estimates for the parabola and the surface $S_6$ give the base of the inductive argument for the constant $D_2$.

\begin{lem}[Base of the induction]\label{basement} For every $R \geq 1$, it holds that
\begin{align}
    D_2(R^{1/2},1,R) &\lesssim_p R^{\frac2p-\frac12},
    \\
    D_2(1,R^{1/2},R) &\lesssim_{p} R^{-\frac13+\frac1p},
\end{align}
for every $p>4$. 
\end{lem}
Note that both of the exponent of $R$ on the right hand side are smaller than zero for $p>4$. This will make it possible to close the induction.

\begin{proof}[Proof of Lemma \ref{basement}]
Let us start with the first inequality. By the definition of $D_2$, we need to prove 
\begin{equation}\label{200901e7.10}
    \|E_{I \times [0,1]^2 }f\|_{L^p(B_R)} \lesssim_p R^{\frac{2}{p}-\frac12} \|f\|_{L^p(\R^3)},
\end{equation}
for every interval $I$ of length $R^{-1/2}$ and ball $B_R\subset \R^5$ of radius $R$.  We may assume that $I=[0,R^{-1/2}]$ and $B_R=[0,R]^5$. By $L^2$ orthogonality and interpolation (see for instance Appendix B of \cite{arxiv:1811.02207}), we can decompose the interval $I$ into smaller intervals with side length $R^{-1}$:
\begin{equation}
    \|E_{I \times [0,1]^2 }f\|_{L^p(B_R)} \lesssim_p 
    R^{\frac12(1-\frac{2}{p})}
    \Big(\sum_{|I'|=R^{-1}; I'\subset I}
    \|E_{I' \times [0,1]^2 }f\|_{L^p(w_{B_R})}^p \Big)^{1/p}.
\end{equation}
Here, $w_{B_R}(x):=(1+|x|/R)^{-500}$ is a weight function. All the terms in the above sum can be handled in the same way. We take $I'=[0, R^{-1}]$ as an example. Recall that the surface under consideration is $(\xi_1, \xi_2, \xi_3, \xi_1 \xi_2+\xi_3^2, \xi_1^2)$. The choice of the interval $I'$ guarantees that $|\xi_1|\le R^{-1}$. By the uncertainty principle, the surface can be thought of as $\{(\xi_1,\xi_2,\xi_3,\xi_3^2,0) \}$. Therefore, by the Hausdorff-Young inequality and well-known $L^p \rightarrow L^p$ restriction estimates for the parabola, we obtain 
\begin{equation}\label{09012e8.23}
    \|E_{I'\times [0, 1]^2} f\|_{L^p(w_{B_R})} \lesim_{p} R^{1/p}\Big\|\|f_{I'}\|_{L^{p'}_{\xi_1, \xi_2}}\Big\|_{L^p_{\xi_3}}\lesim_p R^{\frac 1 p}R^{-(\frac{1}{p'}-\frac 1 p)}\|f_{I'}\|_p,
\end{equation}
where $f_{I'}:=f\cdot 1_{I'\times [0, 1]^2}$ and in the last step we applied H\"older's inequality. We sum over $I'\subset I$ on the right hand side of \eqref{09012e8.23} and further obtain 
\begin{equation}
    \|E_{I\times [0, 1]^2} f\|_{L^p(B_R)}\lesim_p R^{\frac12(1-\frac{2}{p})}
    R^{\frac{1}{p}}
    (R^{-1})^{\frac{1}{p'}-\frac{1}{p}}\|f\|_{L^p([0,1]^3)}.
\end{equation}
This finishes the proof of the first inequality.

\medskip

Let us move on to the second inequality. By the definition of $D_2$, we need to prove 
\begin{equation}\label{200901e7.10zz}
    \|E_{[0,1]^2\times J }f\|_{L^p(B_R)} \lesssim_{p} R^{-\frac13+\frac1p}\|f\|_{L^p(\R^3)},
\end{equation}
for every interval $J$ of length $R^{-1/2}$.  We may assume that $J=[0,R^{-1/2}]$.  By the uncertainty principle, the surface $S$ can be thought of as
$\{(\xi_1,\xi_2,\xi_3,\xi_1^2,\xi_1\xi_2) \}.$
By applying the restriction estimate \eqref{200901e7.7} for such a surface and the Hausdorff-Young inequality, we obtain
\begin{equation}
\begin{split}
    \|E_{[0,1]^2 \times J}f\|_{L^p(B_R)}
    &\lesssim_{p} R^{\frac16}\|f\|_{L^p_{\xi_1,\xi_2}([0,1]^2)L^{p'}_{\xi_3}([0,R^{-1/2}]) }
    \\& \lesssim_{p}
    R^{\frac16}
    R^{-\frac12(\frac{1}{p'}-\frac1p)}\|f\|_{L^p([0,1]^3)},
\end{split}
\end{equation}
where in the last step we applied H\"older's inequality. This completes the proof of the second inequality, thus the proof of the lemma. 
\end{proof}

Another ingredient in the proof of \eqref{d1bound} is a ``refinement'' of a bilinear restriction estimate. We need to obtain a bilinear restriction estimate with favourable  dependence on the supports of input functions. We say that two rectangular boxes $Q_i=[a_i,a_i+l_1] \times [0,1] \times [b_i,b_i+l_2]$ with $i=1, 2$ are \emph{separated} provided that
\begin{equation}
\begin{split}
    |a_1-&a_2| \geq 5 l_1 \;\; \text{and}\;\; |b_1-b_2| \geq 5 l_2.
\end{split}
\end{equation}

\begin{lem}[A bilinear restriction estimate]\label{200901lem7.4}
Let $Q_1$ and $Q_2$ be two rectangular boxes that are given as above and are separated. Then for $p \geq 4$, we obtain
\begin{equation}\label{200905e8.28}
    \||E_{Q_1}f_1E_{Q_2}f_2|^{\frac12}\|_{L^p(\R^5)} \lesssim_p
    (l_1l_2)^{1-\frac{4}{p}}\|f_1\|_{L^p(Q_1)}^{1/2}\|f_2\|_{L^p(Q_2)}^{1/2}.
\end{equation}
\end{lem}

\begin{proof}[Proof of Lemma \ref{200901lem7.4}]
We will prove \eqref{200905e8.28} at $p=4$ and $p=\infty$ separately. The desired estimate at $p=\infty$ is trivial: 
\begin{equation}
    \||E_{Q_1}f_1E_{Q_2}f_2|^{\frac12}\|_{L^{\infty}(\R^5)} \le \|E_{Q_1} f_1\|_{\infty}^{1/2} \|E_{Q_2} f_2\|_{\infty}^{1/2}\le \|f_1\|_1^{1/2} \|f_2\|_1^{1/2}.  
\end{equation}
We apply H\"older's inequality to the right hand side and obtain the desired bound. 

It remains to prove \eqref{200905e8.28} at $p=4$. We write the fourth power of the left hand side of \eqref{200905e8.28} with $p=4$ as 
\begin{equation}\label{200905e8.30}
    \int \norm{\iint_{Q_1\times Q_2} f_1(\xi)f_2(\xi') e(\star)d\xi d\xi'}^2
\end{equation}
where 
\begin{equation}
    \star=x'(\xi+\xi')+x_4(P(\xi)+P(\xi'))+x_5 (Q(\xi)+Q(\xi')),
\end{equation}
with $x'=(x_1, x_2, x_3)$, $\xi=(\xi_1, \xi_2, \xi_3)$, $\xi'=(\xi'_1, \xi'_2, \xi'_3)$, $P(\xi)=\xi_1 \xi_2+\xi_3^2$ and $Q(\xi)=\xi_1^2$. Apply the change of variables 
\begin{equation}
    \begin{split}
        & \xi+\xi'=\eta; \ \ \xi_2=\eta_6;\\
        & P(\xi)+P(\xi')=\eta_4; \ Q(\xi)+Q(\xi')=\eta_5,
    \end{split}
\end{equation}
with $\eta=(\eta_1, \eta_2, \eta_3)$. The Jacobian $J$ of the change of variable is $4|\xi_1-\xi'_1||\xi_3-\xi'_3|$. We obtain that \eqref{200905e8.30} equals 
\begin{equation}
    \int_{\R^5} \norm{\iint \widetilde{f_1} \widetilde{f_2} J^{-1}\cdot e(x\cdot \eta')d \eta' d\eta_6}^2 dx,
\end{equation}
for appropriate $\widetilde{f_1}$ and $\widetilde{f_2}$. We apply H\"older's inequality to the integration in $d \eta_6$ and Plancherel's theorem to the variables $\eta'$, and obtain 
\begin{equation}
    \iint |\widetilde{f_1}|^2 |\widetilde{f_2}|^2 |J|^{-2} \le \iint_{Q_1\times Q_2} |f_1|^2 |f_2|^2 J^{-1},
\end{equation}
where in the last step we changed variables back. So far we have proven that 
\begin{equation}
    \||E_{Q_1}f_1E_{Q_2}f_2|^{\frac12}\|_{L^4(\R^5)}\lesim l_1^{-1/4} l_2^{-1/4} \|f_1\|_2^{1/2}\|f_2\|_2^{1/2}.  
\end{equation}
In the end, we apply H\"older's inequality $L^2\to L^4$ to $f_1$ and $f_2$ and obtain the desired bound. 
\end{proof}

\subsection{Proof of \eqref{d1bound}}

Fix $p>4$. Let $\epsilon=\epsilon_p>0$ be a sufficiently small constant that is to be determined. It will tend to 0 as $p$ tends to 4. Let $K>0$ be a large constant depending on $\epsilon$ that is to be determined. Recall that we need to prove \eqref{d1bound}. The proof is via an induction on $R$. Let us assume that we have proven 
\begin{equation}\label{200903e7.18}
    D_1(R'_0)\le C_p, \text{ for every } R'_0\le R/2.
\end{equation}
To proceed, we need to make use of the constant $D_2$.

\begin{lem}\label{d2d3bound}
There exists $\epsilon=\epsilon_p>0$ such that 
\begin{equation}\label{d2bound}
    D_2(M,1,R') \leq C_{p}M^{-\epsilon}
\end{equation}
for $K \leq M \leq (R')^{1/2}$ and $K^2 \leq R' \leq R$. Similarly, it holds that
\begin{equation}\label{d3bound}
    D_2(1,N,R') \leq C_{p}N^{-\epsilon}
\end{equation}
for $K^{1/2} \leq N \leq (R')^{1/2}$ and $K \leq R' \leq R$. 
\end{lem}

As mentioned above, here $K$ is a large constant depending on $\epsilon$ that is to be determined.

\begin{proof}[Proof of \eqref{d1bound} by assuming Lemma \ref{d2d3bound}]
First of all, we apply a simple version of the broad-narrow analysis of Bourgain and Guth \cite{MR2860188}. In the current situation, as we only need to bi-linearize our Fourier extension operator (instead of multi-linearize), we can directly observe the pointwise estimate 
\begin{equation}\label{200902e7.21}
\begin{split}
    |E_{[0,1]^3}f(x)| &\leq 10 \big(\sum_{a \in 2K^{-1}\Z } |E_{[a,a+2K^{-1}]\times[0,1]^2 }f(x)|^p \big)^{1/p}
        \\&
        +
        10 \big(\sum_{b \in 2K^{-1/2}\Z } |E_{[0,1]^2 \times [b,b+2K^{-1/2}] }f(x)|^p \big)^{1/p}
        \\&
        +
        2K^{10}\sup_{Q_1,Q_2 \text{ separated }}|E_{Q_1}f(x)E_{Q_2}f(x)|^{1/2},
\end{split}
\end{equation}
where $Q_i$ is a rectangular box of the form $[a_i, a_i+ K^{-1}]\times [0, 1]\times [b_i, b_i+K^{-1/2}]$ with $i=1, 2$. 
We raise both sides to the $p$-th power, integrate over $B_R$ and obtain
\begin{equation}\label{200902e7.22}
    \begin{split}
        \|E_{ [0,1]^3}f\|_{L^p(B_R)}
        &\lesssim 
        \Big(
        \sum_{a \in 2K^{-1}\Z }
        \| E_{[a,a+2K^{-1}]\times[0,1]^2 }f \|_{L^p(B_R)}^p\Big)^{1/p}
        \\&
        +
        \Big(
        \sum_{b \in 2K^{-1/2}\Z }
        \| E_{[0,1]^2 \times [b,b+2K^{-1/2}] }f \|_{L^p(B_R)}^p\Big)^{1/p}
        \\&
        +K^{100}
        \sup_{Q_1,Q_2\text{ separated }}
        \|E_{Q_1}fE_{Q_2}f|^{1/2}\|_{L^p(B_R)}.
    \end{split}
\end{equation}
By the definition of $D_2$, Lemma \ref{d2d3bound} and the bilinear restriction estimate in Lemma \ref{200901lem7.4}, the above terms are bounded by
\begin{equation}
 C\Big( C_{p}K^{-\epsilon}+C_{p}(K^{1/2})^{-\epsilon}+   K^{100} \Big) \|f\|_{p},
\end{equation}
for some universal constant $C$. To close the induction, we first pick $K$ (depending on $\epsilon$) such that $C K^{-\epsilon/2}=1/4$, and then pick $C_p=2C K^{100K/\epsilon^4}$. Here we pick an extremely large constant $C_p$ for later use. The choice of it can always be adjusted for different purposes. This finishes the proof of \eqref{d1bound}. 
\end{proof}

\begin{proof}[Proof of Lemma \ref{d2d3bound}]

The proofs for these two estimates are similar, and we will only present the proof of  \eqref{d2bound}. As before, we may assume that $R'$ is a large number compared to $K$. We will prove \eqref{d2bound} via an induction on $R'$ and therefore assume that we have proven 
\begin{equation}\label{200902e7.24}
    D_2(M, 1, R'')\le C_p M^{-\epsilon},
\end{equation}
for $K\le M\le (R'')^{1/2}$ and $R''\le R'/2$. Our goal is to prove 
\begin{equation}\label{d2bound_zzz}
    D_2(M, 1, R')\le C_p M^{-\epsilon}, \text{ for } K\le M\le (R')^{1/2}.
\end{equation}
To prove such an estimate, we will apply a further (backward) induction on $M$. 
 By Lemma \ref{basement} and the triangle inequality, the desired estimate \eqref{d2bound_zzz} holds true for $M \in [K^{-100}(R')^{1/2},(R')^{1/2}]$. Hence, we may assume that $M \leq K^{-100}(R')^{1/2}$ and assume that we have proven 
 \begin{equation}\label{d2bound_zzz_dd}
    D_2(M', 1, R')\le C_p (M')^{-\epsilon}, \text{ for } M'\ge 2M.
\end{equation}
To prove \eqref{d2bound_zzz}, by the definition of $D_2$, we need to prove
\begin{equation}\label{200902e7.27}
        \|E_{I \times [0,1]^2 }f\|_{L^p(B_{R'})} \leq C_{p}M^{-\epsilon}
        \|f\|_{L^p(\R^3)},
\end{equation}
for every interval  $I$ of length $M^{-1}$. We may assume that $I=[0,M^{-1}]$. 

\medskip

To prove \eqref{200902e7.27}, we first write 
\begin{equation}
    E_{I\times [0, 1]^2} f=\sum_{a\in 2(KM)^{-1}\Z}\sum_{b\in 2K^{-1/2}\Z} E_{[a, a+2(KM)^{-1}]\times [0, 1]\times [b, b+2K^{-1/2}]}f.
\end{equation}
By the same argument as how we derived \eqref{200902e7.21} and \eqref{200902e7.22}, we also have 
\begin{equation}
    \begin{split}
        \|E_{I \times [0,1]^2}f\|_{L^p(B_{R'})}
        &\lesssim 
        \Big(
        \sum_{a \in 2(KM)^{-1}\Z }
        \| E_{[a,a+2(KM)^{-1}]\times[0,1]^2 }f \|_{L^p(B_{R'})}^p\Big)^{1/p}
        \\&
        +
        \Big(
        \sum_{b \in 2K^{-1/2}\Z }
        \| E_{[0,M]\times[0,1] \times [b,b+2K^{-1/2}] }f \|_{L^p(B_R)}^p\Big)^{1/p}
        \\&
        +K^{100}
        \sup_{Q_1,Q_2\text{ separated }}
        \||E_{Q_1}fE_{Q_2}f|^{1/2}\|_{L^p(B_{R'})},
    \end{split}
\end{equation}
where $Q_i$ is a rectangular box of the form $[a_i, a_i+2(KM)^{-1}]\times [0, 1]\times [b_i, b_i+2K^{-1/2}]$ and the implicit constant is absolute. We apply the definition of $D_2$ to the first term, apply the bilinear restriction estimate in Lemma \ref{200901lem7.4} to the last term, and obtain
\begin{equation}\label{beforerescaling}
    \begin{split}
        \|E_{I \times [0,1]^2}f\|_{L^p(B_{R'})}
        &\lesssim 
        \big(
        D_2(KM/2,1,R')+K^{200}M^{-\epsilon} \big)
        \|f\|_{L^p}
        \\&+
        \Big(
        \sum_{b \in 2K^{-1/2}\Z }
        \| E_{[0,M^{-1}]\times[0,1] \times [b,b+2K^{-1/2}] }f \|_{L^p(B_{R'})}^p\Big)^{1/p}.
    \end{split}
\end{equation}
By the induction hypothesis \eqref{d2bound_zzz_dd} on $D_2$, the first term on the right hand side will be harmless. The second term will be absorbed by $C_p M^{-\epsilon}$, the desired bound in \eqref{200902e7.27}, as $C_p$ is sufficiently large.  Let us further process the last term. All the terms in the sum there can be handled similarly. For simplicity, we only look at the term $b=0$. To handle this term, we make use of the anisotropic rescaling
\begin{equation}\label{anisotropic}
    (\xi_1,\xi_2,\xi_3) \mapsto (4\xi_1/K,\xi_2,2\xi_3/K^{1/2}),
\end{equation}
and obtain 
\begin{equation}\label{200902e7.32}
\begin{split}
    & \| E_{[0,M^{-1}]\times[0,1] \times [0,K^{-1/2}] }f \|_{L^p(B_{R'})}\\
    & \lesssim
    K^{\frac12(-3+\frac{9}{p})}
    \| E_{[0,4^{-1}KM^{-1}]\times[0,1] \times [0,1] }\tilde{f} \|_{L^p(B_{R'})},
\end{split}
\end{equation}
where $\tilde{f}(\xi_1,\xi_2,\xi_3):=f(4\xi_1/K,\xi_2,2\xi_3/K^{1/2})$ and the implicit constant is absolute.
In order to apply the induction hypothesis \eqref{200902e7.24}, we cut $B_{R'}$ into smaller balls $B_{R'/2}$, and bound \eqref{200902e7.32} by 
\begin{equation}\label{200902e7.33}
\begin{split}
       C K^{\frac12(-3+\frac{9}{p})}K^{\frac{3}{2p}}D_2(4M/K,1,R'/2)\|{f}\|_{L^p}.
\end{split}
\end{equation}
Now if we are in the case $M\ge K^2$, then $4M/K\ge K$ and we can apply the induction hypothesis \eqref{200902e7.24}, and bound the last expression by 
\begin{equation}
\begin{split}
       C C_p K^{\frac12(-3+\frac{9}{p})}K^{\frac{3}{2p}}M^{-\epsilon} K^{\epsilon}\|{f}\|_{L^p}.
\end{split}
\end{equation}
If we are in the case $M\le K^2$, then we simply use \eqref{domination} and \eqref{domination2}, and the induction hypothesis \eqref{200903e7.18} for the constant $D_1$. This allows us to bound \eqref{200902e7.33} by 
\begin{equation}
    C C_p K^{\frac12(-3+\frac{9}{p})}K^{\frac{3}{2p}}\|{f}\|_{L^p}.
\end{equation}
In either case, we can see that if we choose $\epsilon$ to be small enough, depending on how close $p$ is to 4, we always have 
\begin{equation}
    \eqref{200902e7.33}\le C C_p K^{-\epsilon} M^{-\epsilon} \|f\|_p. 
\end{equation}
We apply this bound to \eqref{beforerescaling} and see that the left hand side there can be bounded by 
\begin{equation}
\begin{split}
  (D_2(KM/2,1,R')+K^{200}M^{-\epsilon}+ C C_p K^{-\epsilon} M^{-\epsilon} \big)
        \|f\|_{L^p}.
\end{split}
\end{equation}
We apply to the first term the induction hypothesis \eqref{d2bound_zzz_dd} and obtain 
\begin{equation}
    C\Big( 
    (K/2)^{-\epsilon}
    C_{p}M^{-\epsilon} + K^{200}M^{-\epsilon} +K^{-\epsilon}C_{p}M^{-\epsilon} \Big)\|f\|_{L^p}.
\end{equation}
Since $K$ is large and $R'$ is also large compared to $K$, we can further bound the last display by $C_p  M^{-\epsilon} \|f\|_p$. This finishes the proof of \eqref{200902e7.27} and therefore closes the induction. 
\end{proof}

	\bibliography{reference}{}
	\bibliographystyle{alpha}

\medskip

\medskip

\noindent Department of Mathematics, University of Wisconsin-Madison\\
\emph{Email address}: 
shaomingguo@math.wisc.edu
\\

\noindent Department of Mathematics, University of Wisconsin-Madison\\
\emph{Email address}: coh28@wisc.edu

\end{document}